\documentclass[11pt]{amsart}

\usepackage{amssymb,amsmath,amsthm}
\usepackage{stmaryrd}
\usepackage{enumitem}
\usepackage[all]{xy}

\usepackage{multirow}
\usepackage{tabu}
\setenumerate{label=(\alph*), font=\normalfont}

 \newtheorem{thm}{Theorem}[section]
 \newtheorem{prop}[thm]{Proposition}
 \newtheorem{lem}[thm]{Lemma}
 \newtheorem{cor}[thm]{Corollary}

 \theoremstyle{definition}
 \newtheorem{example}[thm]{Example}
 \newtheorem{defn}[thm]{Definition}
 \newtheorem{remark}[thm]{Remark}

\numberwithin{equation}{section}

\newcommand{\bbZ}{{\mathbb{Z}}}

\newcommand{\bbP}{{\mathbb{P}}}
\newcommand{\bbA}{{\mathbb{A}}}

\newcommand{\bbC}{{\mathbb{C}}}
\newcommand{\bbF}{{\mathbb{F}}}

\newcommand{\Cr}{\operatorname{Cr}}

\newcommand{\GL}{\operatorname{GL}}

\newcommand{\SL}{\operatorname{SL}}

\newcommand{\PGL}{\operatorname{PGL}}
\newcommand{\PSL}{\operatorname{PSL}}

\newcommand{\Aut}{\operatorname{Aut}}

\newcommand{\Spec}{\operatorname{Spec}}

\newcommand{\ed}{\operatorname{ed}}

\newcommand{\im}{\operatorname{im}}

\newcommand{\Pic}{\operatorname{Pic}}

\begin{document}

\title[$G$-unirationality of del Pezzo surfaces (degree 3 and 4)]{Equivariant
unirationality of del Pezzo surfaces of degree 3 and 4}

\author{Alexander Duncan}
\address{
Department of Mathematics, College of Arts and Sciences, University of South Carolina, Columbia, SC, USA}
\email{duncan@math.sc.edu}
\thanks{The author was partially supported by
National Science Foundation
Research Training Grant DMS~0943832.}


\begin{abstract}
A variety $X$ with an action of a finite group $G$ is said to be
$G$-unirational if there is a $G$-equivariant dominant rational map
$V \dasharrow X$ where $V$ is a faithful linear representation of $G$.
This generalizes the usual notion of unirationality.
We determine when $X$ is $G$-unirational for any complex
del Pezzo surface $X$ of degree at least $3$.
\end{abstract}

\maketitle

\section{Introduction}

Recall that a variety $X$ is \emph{unirational} if there exists a
dominant rational map $\bbA^n \dasharrow X$ where $\bbA^n$ is an affine
space.  If a variety $X$ has an action of a finite group $G$,
then $X$ is \emph{$G$-unirational} if there exists a dominant rational
$G$-equivariant map $V \dasharrow X$ where $V$ is a faithful linear
representation of $G$.
When $G$ is trivial, the linear representation is just an affine space
and we recover the usual notion of unirationality.

One application of $G$-unirationality is the construction of
\emph{versal} or \emph{generic} objects in algebra and number theory.
For example, consider the following classical result of
Hermite~\cite{Her61Sur-linvariant}.
For any separable field extension $L/K$ of degree $5$ of a field $K$ of
characteristic not $2$, there exists a generator $x \in L$ whose minimal
polynomial has the form
\[ x^5 + bx^3 + cx + c = 0 \]
where $b$ and $c$ are elements in $K$
(see~\cite{Kra06A-result} for a modern exposition
and \cite{Cor76Algebraic} for an alternate proof).
The original proof reduces to the following:

\begin{thm}[Hermite] \label{thm:Hermite}
Let $X$ be the Clebsch diagonal cubic surface given by
\begin{equation} \label{eq:Clebsch}
\sum_{i=1}^5 x_i = \sum_{i=1}^5 x_i^3 = 0
\end{equation}
in $\bbP^4$, and let $S_5$ act on $X$ by permutations of the coordinates
$x_1, \ldots, x_5$.  There exists a dominant rational $S_5$-equivariant
map $V \dasharrow X$ where $V$ is a linear representation of $S_5$.
In other words, $X$ is $S_5$-unirational.
\end{thm}

Both Hermite and Coray's proofs take advantage of special features of
the Clebsch surface which do not generalize to other surfaces.
In this paper, we characterize equivariant unirationality for all
complex del Pezzo surfaces of degree $d \ge 3$.
In particular, we obtain a new proof of
Theorem~\ref{thm:Hermite}.

Before stating the main theorem, we discuss connections between
the equivariant and arithmetic notions of unirationality.
The geometric action of the group $G$ is analogous to the
arithmetic action of the absolute Galois group of the base field.
In~\cite{DunRei15Versality}, this analogy is made precise
(using the term ``$G$-very versal'' instead of ``$G$-unirational'').
One can transform results about arithmetic unirationality into
results about equivariant unirationality, and vice versa.
Consequently, while our main focus in this paper is the base field $\bbC$,
most of our constructions use, or are inspired by, arithmetic results.

In the arithmetic situation, if a smooth variety is
$k$-unirational, then $X$ has a rational $k$-point.
For smooth geometrically unirational varieties, it is still an open question
whether the converse is true.
However, in particular we have the following
from Theorems 29.4 and 30.1 of~\cite{Man86Cubic}:

\begin{thm}[Manin] \label{thm:Manin}
Let $X$ be a del Pezzo surface of degree $d \ge 3$ over a field $k$
of characteristic $0$.
Then $X$ is $k$-unirational if and only if $X$ has a rational $k$-point.
\end{thm}

\begin{remark}
The assumption of characteristic $0$ is unnecessary in the above theorem,
but this is all we need.  Manin's proof also applies when
$k$ has enough elements or when $d \ge 5$.
The remaining cases have since been settled;
see \cite{Kol02Unirationality} for $d=3$,
and \cite{Pie12On-the-unirationality} or \cite{Kne15Degree}
for $d = 4$.
\end{remark}

Now, consider a del Pezzo surface $X$ of degree $d \ge 3$ with a faithful
$G$-action.
Naively, one might expect that the corresponding equivariant result
would be that $X$ is $G$-unirational if and only if it has a $G$-fixed
point.
Indeed, using the machinery of~\cite{DunRei15Versality},
one can show that the existence of a $G$-fixed point implies that $X$ is
$G$-unirational (see Corollary~\ref{cor:fix2Guni}).
However, the other direction of the naive analog fails in general.
For example, the Clebsch cubic~\eqref{eq:Clebsch} is $S_5$-unirational
but has no $S_5$-fixed points.

Nevertheless, if one assumes that $G$ is abelian then $X$ is
$G$-unirational if and only if $X$ has a $G$-fixed point
(see Corollary~\ref{cor:abelSubgroups}).
This provides an obstruction for $G$-unirationality: every abelian
subgroup must have a fixed point.
This turns out to be the only obstruction for del Pezzo surfaces of
degree $d \ge 3$.

\begin{thm} \label{thm:main}
Let $X$ be a del Pezzo surface of degree $d \ge 3$ with a faithful
$G$-action.
Then the following are equivalent.
\begin{enumerate}
\item $X$ is $G$-unirational.
\item $X$ has an $A$-fixed point for all abelian subgroups $A$ in $G$.
\item $X$ is $G_p$-unirational for all Sylow $p$-subgroups $G_p$
and all primes $p$.
\item $X$ is $G_p$-unirational for all Sylow $p$-subgroups $G_p$
and all primes $p$ dividing $d$.
\end{enumerate}
\end{thm}

\begin{remark}
For del Pezzo surfaces of degree $d \ge 3$, being $G$-unirational is
equivalent to being $G$-versal (see Proposition~\ref{prop:wv2Guni}).
We emphasize $G$-unirationality in this paper because it is a stronger
property and is more geometric.
\end{remark}

\begin{remark}
For $d \ge 6$, Theorem~\ref{thm:main} is obtained by revisiting the
classification of finite groups of essential dimension $2$
in~\cite{Dun13Finite}.  The case of $d = 5$ is implicit in the proof of
Theorem~6.5(c)~of~\cite{BuhRei97On-the-essential}.
Other results on the $G$-versality (and hence $G$-unirationality) of
these surfaces can be found in
\cite{Tok04Note,Tok052-dimensional,Tok06Two-dimensional,
Ban07Construction,Ban08Versal}.
\end{remark}

\begin{remark}
Condition (c) is closely related to the existence of a $0$-cycle of
degree $1$ in the arithmetic setting (see
Section~8~of~\cite{DunRei15Versality}).
One might ask if the existence of a $0$-cycle of degree $1$ implies
the existence of a rational point,
but this is false even for geometrically rational surfaces
(see~\cite{ColCor79Lequivalence}).
However,
for a del Pezzo surface $X$ of degree $d \ge 4$ over a perfect field $k$,
there is a rational $k$-point as soon as there is a rational $L$-point
for a finite extension $L/k$ of degree prime to $d$
(see~\cite{Cor77Points}).
Using this fact, one can prove the equivalence of conditions (a), (c), and (d)
in Theorem~\ref{thm:main} for $d \ge 4$, however one cannot conclude
anything about condition (b).

For cubic hypersurfaces (in particular, for del Pezzo surfaces of degree
$3$) it was conjectured by Cassels and Swinnerton-Dyer that
the existence of a $0$-cycle of degree coprime to $3$ implies
the existence of a rational point (see~\cite{Cor76Algebraic}).
Theorem~10.5~of~\cite{DunRei15Versality} shows that, for $d=3$,
the equivalence of conditions (a), (c), and (d) in
Theorem~\ref{thm:main} follows from this conjecture.
Consequently, Theorem~\ref{thm:main} can be viewed as evidence for the
conjecture.
\end{remark}

\begin{remark}
For degree $d=2$, a version of Theorem~\ref{thm:Manin} is known to
apply when the rational point lies outside of a certain closed subscheme
(see \cite{SalTesVar14On-the-uniratio,FesLui16Unirationality}).
However, the fixed points of $G$-actions often lie
on this subscheme
(see cases 2A and 2B of Theorem~1.1~of~\cite{DolDun14Fixed}).
Theorem~\ref{thm:main} does not hold for $d=2$, as the following example
shows.
\end{remark}

\begin{example}
Consider the del Pezzo surface $X$ of degree $2$ given by
\[
x_4^2 = x_1^3x_2 + x_2^3x_3 + x_3^3x_1
\]
in the weighted projective space $\bbP(1:1:1:2)$,
which has automorphism group $C_2 \times \PSL_2(\bbF_7)$
(see Table~8.9~of~\cite{Dol12Classical}).
In particular, $X$ has a faithful action of a finite group
$G \simeq C_2 \times (C_7 \rtimes C_3)$.
One checks that every abelian subgroup $A$ of $G$ has a fixed point on $X$,
so condition (b) from Theorem~\ref{thm:main} holds.

However, the essential dimension of $G$ is greater than $2$
(see Definition~\ref{def:ed}).
This follows by Lemma~7.2~and~Theorem~3.1~of~\cite{KraLotSch09Compression}.
We conclude that $X$ is not $G$-unirational and
that Theorem~\ref{thm:main} fails for $d=2$.
\end{example}

\begin{remark}
For degree $d=1$, there is always a canonical point on $X$.
In particular, if Theorem~\ref{thm:Manin} were true in this case
then every del Pezzo $G$-surface of degree $1$ would be $G$-unirational
and Theorem~\ref{thm:main} would be true in this case as well.
However, Theorem~\ref{thm:Manin} is completely open for minimal surfaces
of degree $1$.
\end{remark}

In Section~\ref{sec:prelims}, we recall the structure theory of del
Pezzo surfaces and some useful facts about equivariant unirationality.
In Section~\ref{sec:Hermite}, we prove that
$G$-unirationality of a cubic surface reduces to consideration of
$H$-unirationality where $H$ is a subgroup of $G$ of index $2$.
We then show how this can be used to reprove Theorem~\ref{thm:Hermite}.
In Section~\ref{sec:obstructions}, we identify four families of actions
by elementary abelian groups on del Pezzo surfaces which do not have
fixed points.
These will turn out to be the only obstructions to equivariant unirationality
for degree $\ge 3$ (see Theorem~\ref{thm:obsMain}).
In Sections~\ref{sec:deg5to9}--\ref{sec:deg3}, we prove
Theorem~\ref{thm:main} by proving Theorem~\ref{thm:obsMain}.

\section{Preliminaries}
\label{sec:prelims}

Throughout, a \emph{$k$-variety} is a geometrically integral scheme of
finite type over a field $k$ of characteristic $0$.
A $k$-variety $X$ is $k$-unirational if there exists a dominant rational map
$\bbA^n_k \dasharrow X$ defined over $k$.
A \emph{variety}, \emph{surface}, or \emph{curve}, without explicit
reference to a base field, will have base field $k=\bbC$.

\subsection{Del Pezzo surfaces}

A del Pezzo surface $X$ is a smooth projective surface whose
anticanonical class $-K_X$ is ample.
The degree $d=K_X^2$ of a del Pezzo surface is an integer $1 \le d \le 9$.
Except for $\bbP^1 \times \bbP^1$ in degree $8$, every del Pezzo surface
is isomorphic to $\bbP^2$ blown up at $9-d$ points.

For degree $d \le 5$, the automorphism group $\Aut(X)$ of $X$ induces
a faithful action on $\Pic(X)$.
In fact, there is an injective homomorphism
\begin{equation} \label{eq:Weyl}
\Aut(X) \hookrightarrow W(E_n)
\end{equation}
where $W(E_n)$ is the Weyl group of a simple root system of type $E_n$
(by convention, $E_5 = D_5$, $E_4 = A_4$).

Let $X$ be a smooth projective surface with a faithful action of a
finite group $G$.
We say that $X$ is a \emph{minimal} $G$-surface if any equivariant birational
morphism $X \to X'$ is an isomorphism, where $X'$ is another $G$-surface.
By blowing down $G$-stable sets of skew $(-1)$-curves, every
rational $G$-surface is equivariantly birationally equivalent to
a minimal rational $G$-surface.
By \cite{Man67Rational} and \cite{Isk79Minimal},
all minimal rational $G$-surfaces are either
\begin{itemize}
\item del Pezzo $G$-surfaces with $\Pic(X)^G \simeq \bbZ$, or
\item conic bundle $G$-surfaces, where there is an equivariant
morphism to $\bbP^1$ with rational general fiber and $\Pic(X)^G \simeq
\bbZ^2$.
\end{itemize}

\subsection{Equivariant Unirationality}

Let $X$ be a variety with an action of a finite group $G$.
The variety $X$ is \emph{$G$-weakly versal} if,
for every \emph{faithful} $G$-variety $Y$, there exists an
equivariant rational map $Y \dasharrow X$.
The variety $X$ is \emph{$G$-versal} if, for every non-empty
$G$-invariant open subset $U$ of $X$, the variety $U$ is $G$-weakly
versal.

The primordial example of a $G$-versal variety is a linear
representation of $G$.
Thus, for $X$ to be $G$-weakly versal, it suffices to find just one
equivariant rational map $V \dasharrow X$ where $V$ is a faithful linear
representation.
In particular, $X$ is $G$-weakly versal as soon as it has a $G$-fixed
point.

Following \cite{DunRei15Versality},
we have the following series of implications:
\[
\textrm{$G$-unirational}
\implies
\textrm{$G$-versal}
\implies
\textrm{$G$-weakly versal}
\]
which correspond to the implications:
\[
\textrm{$k$-unirational}
\implies
\textrm{$X(k)$ Zariski-dense in $X$}
\implies
\textrm{$X(k) \ne \emptyset$}
\]
in the arithmetic setting.

In our context, these three things are all equivalent:

\begin{prop} \label{prop:wv2Guni}
Let $X$ be a del Pezzo surface of degree $\ge 3$ with an action of
a finite group $G$.  The following are equivalent:
\begin{enumerate}
\item $X$ is $G$-unirational,
\item $X$ is $G$-versal,
\item $X$ is $G$-weakly versal.
\end{enumerate}
\end{prop}

\begin{proof}
We note that the property of being a del Pezzo surface of a given degree
is a geometric property.  In other words, all twisted forms of a given
del Pezzo surface of degree $d$ are also del Pezzo surfaces of degree
$d$.
The proposition now follows by using Theorem~\ref{thm:Manin}
in combination with Theorem~1.1~of~\cite{DunRei15Versality}.
\end{proof}

A very useful corollary of this proposition is the following:

\begin{cor} \label{cor:fix2Guni}
Let $X$ be a del Pezzo surface of degree $\ge 3$ with an action of
a finite group $G$.
If $X$ has a $G$-fixed point, then $X$ is $G$-unirational.
\end{cor}

The following well-known fact is a consequence of the holomorphic
Lefschetz fixed-point formula (see \S 3.5.1 of \cite{Ser08Le-groupe}).

\begin{prop} \label{prop:ratCyclic}
If $X$ is a complete smooth rational surface with an action of a
finite cyclic group $G$, then $X$ has a $G$-fixed point.
\end{prop}

\begin{cor} \label{cor:cyc2Guni}
Let $X$ be a del Pezzo surface of degree $\ge 3$ with an action of
a finite cyclic $G$.
Then $X$ is $G$-unirational.
\end{cor}

The following proposition is a useful method for producing fixed points
(Proposition~A.2~of~\cite{ReiYou00Essential}):

\begin{prop}[Going-Down] \label{prop:GoingDown}
Suppose $X \dasharrow Y$ is a $G$-equivariant rational map
of $G$-varieties where $G$ is abelian.
If $X$ has a smooth $G$-fixed point and $Y$ is proper,
then $Y$ has a $G$-fixed point.
\end{prop}

In particular, Proposition~\ref{prop:GoingDown} implies that
the existence of a $G$-fixed point is an equivariant
birational invariant of smooth, proper $G$-varieties.
Since a linear representation always has a smooth fixed point, we have
the following:

\begin{cor} \label{cor:abelSubgroups}
If $X$ is a proper $G$-unirational variety and $G$ is abelian,
then $X$ has a $G$-fixed point.
\end{cor}

One might wonder how much the definition of $G$-unirationality depends
on the particular choice of linear representation $V$.
In general, a $G$-unirational variety might have larger dimension than
some faithful linear representation, so one cannot simply take any
representation.
However, the following consequence of the No-Name Lemma
(see~\cite{Dom08Covariants}) shows that this is not a significant
problem.

\begin{prop} \label{prop:NoName}
If $X$ is $G$-unirational, then for any faithful linear representation
$V$ of $G$, there exists a $G$-equivariant dominant rational map
$V \times \bbA^n \dasharrow X$ for some affine space $\bbA^n$ where
$G$ acts trivially on $\bbA^n$.
\end{prop}

Finally, it is useful to point out the following definition:

\begin{defn} \label{def:ed}
The \emph{essential dimension} of a finite group $G$,
denoted $\ed(G)$,
is the minimal dimension of a faithful $G$-unirational variety.
\end{defn}

The finite groups of essential dimension $2$ were classified
in Theorem 1.1~of~\cite{Dun13Finite}.
Thus, we know the groups $G$ for which a $G$-unirational surface exists,
but we do now know whether a \emph{given} $G$-surface is
$G$-unirational.

\section{A construction for cubic surfaces}
\label{sec:Hermite}

The following is a well-known fact about arithmetic cubic hypersurfaces
(see, for example, Proposition~2.2~of~\cite{Cor76Algebraic}).

\begin{prop} \label{prop:cubicQuad}
Let $X$ be a cubic hypersurface over a field $k$.
If $X$ has a $K$-point for a quadratic extension field $K/k$,
then $X$ has a $k$-point.
\end{prop}

By Theorem~\ref{thm:Manin}, $k$-unirationality of a cubic surface (a del
Pezzo surface of degree $3$) is equivalent to the existence of a smooth
rational $k$-point, so we have the following equivariant analog:

\begin{thm} \label{thm:cubicIndex2}
Let $X$ be a smooth cubic surface with a faithful action of a finite
group $G$.  Suppose $H$ is a subgroup of $G$ of index $2$.
If $X$ is $H$-unirational then $X$ is $G$-unirational.
\end{thm}

\begin{proof}
The following proof was suggested to the author by Z.~Reichstein.
We use the machinery of \cite{DunRei15Versality}.
Let $T \to \Spec(K)$ be a $G$-torsor for some extension field $K/\bbC$.
By Theorem~1.1~of~\cite{DunRei15Versality},
we want to show that
the twisted variety ${}^T X$ is $K$-unirational.
By Theorem~\ref{thm:Manin}, showing that ${}^T X$ has a $K$-point is
sufficient.
The inclusion $H \subset G$ induces an exact sequence
\[
H^1(K,H) \to H^1(K,G) \to H^1(K,C_2)
\]
in Galois cohomology.
Since the image of $T$ in $H^1(K,C_2)$ is split over a field extension
$L/K$ of degree $1$ or $2$,
the torsor $T_L$ descends to an $H$-torsor $S \to \Spec(L)$.
In particular, $({}^{T} X)_L \simeq {}^{T_L} X \simeq {}^S X$.
By our assumption, we know that ${}^S X$ is $L$-unirational for any
$H$-torsor $S \to \Spec(L)$.
Thus, $({}^{T} X)_L$ has an $L$-point.
By Proposition~\ref{prop:cubicQuad},
we conclude that ${}^{T} X$ has a $K$-point as desired.
\end{proof}

\begin{remark}
Note that the same proof also works for cubic hypersurfaces
in higher dimensions under some mild technical hypotheses
(see Theorem~10.5~of~\cite{DunRei15Versality}).
\end{remark}

Theorem~\ref{thm:Hermite} follows quite easily from
Theorem~\ref{thm:cubicIndex2}.
However, one can extract a purely geometric argument without explicit
use of Galois cohomology, which we believe is of independent interest.
The following proof does not rely on \cite{DunRei15Versality}.

\begin{proof}[Proof of Theorem~\ref{thm:Hermite}]
There are two Galois conjugate $3$-dimensional faithful representations
of $A_5$; pick one of them.
Projection from the origin produces a dominant rational $A_5$-equivariant map
\[ \alpha : \bbC^3 \dasharrow \bbP^2 \]
and
thus, by construction, $\bbP^2$ is $A_5$-unirational.
There is a unique $A_5$-orbit containing exactly $6$ points in $\bbP^2$.
Blowing up these points gives an $A_5$-equivariant birational map
\[ \beta : \bbP^2 \dasharrow X \]
to a smooth cubic surface $X$.
The composition $\psi = \beta \circ \alpha$ shows that $X$ is
$A_5$-unirational.

The Clebsch diagonal cubic surface is the only smooth cubic surface with a
faithful $A_5$-action; indeed, it is the only one with a $C_5$-action
(see Theorem~9.5.8~of~\cite{Dol12Classical}).
One can also see this directly by computing the invariants of degree $3$
of the $4$-dimensional representations of $A_5$.
In any case, the surface $X$ described above must be the Clebsch.

At this point, we can conclude that $X$ is $S_5$-unirational
by Theorem~\ref{thm:cubicIndex2}.
Instead, we will supply a more elementary argument.
Only the very last step requires the group to be $S_5$,
so until then we will use $H = A_5$ and $G = S_5$
to emphasize the parallels with Theorem~\ref{thm:cubicIndex2}.

By Proposition~\ref{prop:NoName}, we may assume that we have an
$H$-equivariant dominant rational map
\[ \psi : V \dasharrow X \]
where $V$ is a faithful linear representation of $G$
(this is a linear representation of $H$ by restriction).
We want to construct a $G$-equivariant dominant rational map $V
\dasharrow X$.

Let $\sigma \in G$ be any element such that $\sigma \notin H$.
Since $H$ has index $2$ in $G$, $\sigma^2 \in H$ and $G$ is a disjoint
union of the cosets $H$ and $\sigma H$.
Define a $G$-action on $X \times X$ via
\[
g(x,y) = \begin{cases}
(gx,gy) &\textrm{if } g \in H\\
(gy,gx) &\textrm{if } g \in \sigma H
\end{cases}
\]
where $g \in G$ and $(x,y) \in X \times X$.
The $G$-variety $X \times X$ can be viewed as an equivariant
analog of the Weil restriction of a quadratic extension.
Note that the diagonal embedding $X \hookrightarrow X \times X$
is $G$-equivariant.

We now construct a $G$-equivariant map $\tau : V \dasharrow X \times X$
using $\psi$.
We define
\[ \tau(v):=\left(\sigma\psi(\sigma^{-1}v),\psi(v)\right) \]
for every $v \in V$ which is in the domain of definitions of
both $\psi$ and $\psi \circ \sigma^{-1}$.
We claim that $\tau$ is $G$-equivariant.
Indeed, for $h \in H$, we have
\[ \tau(hv) = (\sigma\psi(\sigma^{-1}hv),\psi(hv))
 = (\sigma\sigma^{-1}h\sigma\psi(\sigma^{-1}v),h\psi(v))
= h\tau(v) \]
and, for $\sigma$, we have
\[ \tau(\sigma v) = (\sigma\psi(v),\psi(\sigma v))
= \sigma\tau(v) \ . \]
We have checked equivariance on the generating set, thus we
may conclude that $\tau$ is $G$-equivariant.

Consider the
``third intersection map'' $\omega : X \times X \dasharrow X$ which takes a
general pair of points $(x,y)$ to the unique third point on $X$
lying on the line through $x$ and $y$ in $\bbP^n$.
The map $\omega$ is a well-defined dominant rational map by Lemma 3.2 of
\cite{Kol02Unirationality}.
Note that since $G$ carries the line through $x$ and $y$ to the line
through $g(x)$ and $g(y)$, we see that $\omega$ is $G$-equivariant.

We would like to form the composition $\omega \circ \tau$
so we need to check that it is well-defined.
The indeterminacy locus of $\omega$ contains only points
$(x,y) \in X \times X$ such that either the line $\overline{xy}$ is
undefined or the line $\overline{xy}$ lies on $X$.
The first case corresponds to the diagonal subvariety
$X \subset X \times X$,
while the second case occurs only when both $x$ and $y$ lie on one of
the 27 lines.

If the image of $\tau$ is contained in the diagonal, then $\psi$ is
already $G$-equivariant and we are done.  Thus, we may assume that
the image of $\tau$ is not contained in the diagonal.
Since $\psi$ is dominant, $\tau$ is dominant on each of the factors of
$X \times X$.  Thus the image of $\tau$ contains a point $(x,y)$
where $x \ne y$ and $x$ is not on one of the 27 lines.
In particular, the image of
$\tau$ intersects the domain of definition of $\omega$ non-trivially.
Thus we have a rational $G$-equivariant map
$\psi' = \omega \circ \tau : V \dasharrow X$.

It remains to show that $\psi' = \omega \circ \tau$ is dominant.
For this last argument, we require $G = S_5$.
First, note $\im(\psi')$ cannot be a point since
$X$ has no $S_5$-fixed points.
If $\im(\psi')$ is a curve then it must be birational to $\bbP^1$.
But $\bbP^1$ does not carry a faithful action of $S_5$ and the
non-faithful actions (either trivial or via an involution)
have fixed points.
Thus $\psi'$ is dominant since the image must have dimension $2$.
\end{proof}

\section{Obstructions to Equivariant Unirationality}
\label{sec:obstructions}

In this section, we discuss some examples of $G$-actions on del Pezzo
surfaces $X$ where $G$ is abelian, but $G$ has no fixed points on $X$.
These surfaces are not $G$-unirational in view
of Corollary~\ref{cor:abelSubgroups}

The plane Cremona group $\Cr(2)$ is the set of birational automorphisms
of a rational surface.
For every finite subgroup $G$ of $\Cr(2)$, there exists a smooth proper
surface $X$ with a $G$-action (indeed, we can assume $X$ is a del Pezzo
surface or a conic bundle);
see Theorem~3.6~of~\cite{DolIsk09Finite}.
Since $G$-unirationality is an equivariant birational invariant,
it makes sense to ask about $G$-unirationality for a finite
subgroup $G$ of the Cremona group without specifying the surface.
We may find the following groups in the classifications of
Blanc~\cite{Bla06Finite}, and
Dolgachev~and~Iskovskikh~\cite{DolIsk09Finite}.

Any subgroup $G$ of the plane Cremona group which contains one of the
groups described in the following examples cannot be $G$-unirational.
Consequently, we may view them as obstructions to $G$-unirationality.

\begin{example}[Obstruction A]
Consider $X \simeq \bbP^1 \times \bbP^1$ and
$G = \langle g_1, g_2 \rangle \simeq C_2^2$ with action:
\begin{align*}
g_1 : (x_1:x_2)\times(y_1:y_2) &\mapsto (x_2:x_1)\times(y_1:y_2)\\
g_2 : (x_1:x_2)\times(y_1:y_2) &\mapsto (x_1:-x_2)\times(y_1:y_2)
\end{align*}
where $(x_1:x_2)\times(y_1:y_2)$ are coordinates on $X$.
There are no $G$-fixed points on $X$.
Up to conjugacy in $\Aut(X)$, there are two other actions of $G$ on
$\bbP^1 \times \bbP^1$ which do not have fixed points.
However, all three of these subgroups correspond to one conjugacy class
in $\Cr(2)$, denoted P1.22.1 using Blanc's notation.
\end{example}

\begin{example}[Obstruction B]
Consider $X \simeq \bbP^2$ and
$G = \langle g_1, g_2 \rangle \simeq C_3^2$ with action:
\begin{align*}
g_1 : (x_1:x_2:x_3) &\mapsto (x_2:x_3:x_1)\\
g_2 : (x_1:x_2:x_3) &\mapsto (x_1:\epsilon x_2:\epsilon^2 x_3)
\end{align*}
where $\epsilon$ is a primitive third root of unity.
This surface has no $G$-fixed points.
Its conjugacy class in $\Cr(2)$ is denoted 0.V9 in Blanc's notation.
\end{example}

Before describing the next two examples, we prove the following lemma.

\begin{lem} \label{lem:genus1fixed}
An automorphism $g$ of a smooth genus 1 curve $E$ has a fixed point if and
only if $g$ is not a non-trivial translation.
\end{lem}

\begin{proof}
Pick an origin for the group law making $E$ an elliptic curve.
Recall that the automorphism group of an elliptic curve
is a semidirect product of the translation group with
the group of group automorphisms.
Indeed, every automorphism $h$ has the form $h(x)=g(x)+a$ where $g$ is a
group automorphism of $E$ and $a$ is an element of $E$.
To find the fixed points we set $h(x)=x$ and solve $x-g(x)=a$ for $x$.
If $g$ is trivial, then $h$ is a translation.
Otherwise, the map $x \mapsto x-g(x)$ is non-constant
and $x-g(x)=a$ has a solution.
\end{proof}

\begin{example}[Obstruction C]
Here $X$ is a del Pezzo surface of degree $4$ and
$G = \langle g_1, g_2 \rangle \simeq C_2^2$ acts on $X$
where the fixed points of $g_1$ lie on an elliptic curve $E$ and
$g_2$ acts on $E$ as translation by a $2$-torsion element.
Explicitly,
$X$ is given by two equations in $\bbP^4$:
\[
x_1^2 + \cdots + x_5^2 = a_1x_1^2 + \cdots a_5x_5^2 = 0
\]
for distinct parameters $a_1, \ldots, a_5$, with
\begin{align*}
g_1(x_1:x_2:x_3:x_4:x_5) &= (x_1:x_2:x_3:x_4:-x_5)\\
g_2(x_1:x_2:x_3:x_4:x_5) &= (-x_1:-x_2:x_3:x_4:x_5) \ .
\end{align*}
These examples form a non-empty family of conjugacy classes denoted C2,2
in Blanc's notation.
This example will be studied in Section \ref{sec:deg4}.
\end{example}

\begin{example}[Obstruction D]
Here $X$ is a del Pezzo surface of degree $3$ and
$G = \langle g_1, g_2 \rangle \simeq C_3^2$ acts on $X$
where the fixed points of $g_1$ lie on an elliptic curve $E$ and
$g_2$ acts on $E$ as translation by a $3$-torsion element.
Explicitly,
$X$ is given by one equation in $\bbP^3$:
\[
x_1^3 + x_2^3 + x_3^3 + x_4^3 + \alpha x_1x_2x_3 = 0
\]
for a parameter $\alpha$, with
\begin{align*}
g_1(x_1:x_2:x_3:x_4) &= (x_1:x_2:x_3:\epsilon x_4)\\
g_2(x_1:x_2:x_3:x_4) &= (x_3:x_1:x_2:x_4)
\end{align*}
where $\epsilon$ is a primitive third root of unity.
These examples form a non-empty family of conjugacy classes denoted 3.33.2
in Blanc's notation.
This example will be studied in Section \ref{sec:deg3}.
\end{example}

For any rational $G$-surface $X$, we say that $X$
\emph{has obstruction A (resp. B, C, D)} if there exists a subgroup $H$
of $G$ such that $X$ is $H$-equivariantly birationally equivalent
to an $H$-surface in one of the corresponding examples above.
Equivalently, considering $G$ has a subgroup of the Cremona group,
we say $G$ has a given obstruction if it contains a subgroup
conjugate to one of those in the above examples.

To prove Theorem~\ref{thm:main},
we will instead prove the following:

\begin{thm} \label{thm:obsMain}
Let $X$ be a del Pezzo surface of degree $d \ge 3$ with a faithful
$G$-action.
Then $X$ is $G$-unirational if and only if it satifies the condition
given in Table~\ref{tab:obsTable}.
\end{thm}

\begin{table}[h]
\begin{tabular}{|c|c|c|}
\hline
Degree & Form & $G$-unirational \\
\hline
\hline
9 & $\bbP^2$ & no obstruction B \\
\hline
8 & $\bbP^1 \times \bbP^1$ & no obstruction A \\
 & $\mathbb{F}_1$ & always \\
\hline
7 &  & always \\
\hline
6 &  & no obstruction A or B \\
\hline
5 & $\overline{\mathcal{M}}_{0,5}$ & always \\
\hline
4 &  & no obstruction A or C \\
\hline
3 & cubic & no obstruction B or D \\
\hline
\end{tabular}
\caption{Obstructions to equivariant unirationality}
\label{tab:obsTable}
\end{table}

Note that in Theorem~\ref{thm:main}, (a) implies all the other
statements.
All of the statements imply the lack of obstructions found in
Table~\ref{tab:obsTable}.
Thus, Theorem~\ref{thm:obsMain} implies Theorem~\ref{thm:main}.
To prove Theorem~\ref{thm:obsMain}, we must show that the
obstructions are the only obstructions to $G$-unirationality.

\section{Degree $d \ge 5$}
\label{sec:deg5to9}

\begin{proof}[Proof of Theorem~\ref{thm:obsMain} for $d \ge 6$.]

Aside from $\bbP^1 \times \bbP^1$ there is only one surface of degree
$8$; it is $\bbP^2$ with a single point blown up.
Blowing down the resulting exceptional divisor is an equivariant
operation which provides a $G$-fixed point on $\bbP^2$
so this surface is always $G$-unirational.

In degree $7$, we have $\bbP^2$ blown up at two
points.  The strict transform of the line between these two points can
be blown down to a $G$-fixed point on $\bbP^1 \times \bbP^1$, so this is
always $G$-unirational as well.
It remains only to consider the surfaces $\bbP^2$, $\bbP^1 \times
\bbP^1$ and the del Pezzo surface of degree $6$

All of these surfaces are toric varieties.
By Corollary~3.6~of~\cite{Dun13Finite},
we see that (a) and (c) of Theorem~\ref{thm:main} are equivalent
(recall that $G$-versal and $G$-unirational are equivalent here by
Proposition~\ref{prop:wv2Guni}.)
In particular, we may assume that $G$ is a $p$-group.

By Proposition~3.10~of~\cite{Dun13Finite}, any $p$-group acts on a smooth
toric surface as a subgroup of $(\bbC^\times)^2 \rtimes H$ where $H$ is a
finite subgroup of $\GL_2(\bbZ)$.
Since any finite subgroup of $\GL_2(\bbZ)$ has order divisible by $2$ or
$3$, if $G$ is a $p$-group where $p \ge 5$ then $G$ is a subgroup
of $(\bbC^\times)^2$.  We conclude that $X$ is always $G$-unirational
for $p \ge 5$ by Lemma~3.8~of~\cite{Dun13Finite}.
Thus, it suffices to assume that $G$ is a $2$-group or a $3$-group.
In either of these cases, a del Pezzo surface of degree $6$ is not
$G$-minimal (there is always a $G$-orbit of exceptional lines that can
be blown down).
Thus we are reduced to the cases of $\bbP^2$ and $\bbP^1 \times \bbP^1$.

First, we consider $X = \bbP^2$.
Any irreducible linear representation of a $2$-group has degree a power of $2$.
Thus, any $3$-dimensional representation of a $2$-group has a
$1$-dimensional subrepresentation.  Thus, any $2$-group $G$ acting on
$\bbP^2$ has a fixed point and, thus, is $G$-unirational.
We may assume that $G$ is a $3$-group acting on $\bbP^2$;
and thus the preimage $\widetilde{G}$ of $G$ in $\GL_3(\bbC)$
acts by monomial matrices.
Note that $\widetilde{G}$ is abelian if and only if it is
diagonalizable; in this case $G$ has a fixed point and $X$ is $G$-unirational.
Thus, there must be a non-diagonal element $g$ in $\widetilde{G}$ of the form
\[
\begin{pmatrix}
0 & a & 0\\
0 & 0 & b\\
c & 0 & 0
\end{pmatrix}
\]
where $a,b,c \in \bbC^\times$.  After possibly choosing a different representative
with the same image in $\PGL_3(\bbC)$, we may assume $g$ has order $3$.
Then a (diagonal) change of coordinates allows us to assume $g$ is a
permutation matrix.
Note that $g$ does not commute with any diagonal matrix except scalar
matrices.
Thus there is a non-scalar diagonal matrix $h$ in $\widetilde{G}$ whose image
in $G$ has order $3$.
Up to permutation or picking a different representative for the same
element in $G$, we see that $h$ has diagonal entries $(1,\epsilon,\epsilon^2)$
or $(1,1,\epsilon)$.
The first choice means that the image of $\langle g, h \rangle$ is
Obstruction B; the second choice means that it
contains Obstruction B.
Since all other cases are $G$-unirational, we've proven the theorem for
$\bbP^2$.

Now, we consider $X = \bbP^1 \times \bbP^1$.
Here any $3$-group must be a subgroup of the torus $(\bbC^\times)^2$,
and thus is always $G$-unirational.
It suffices to assume $G$ is a $2$-group.
The automorphism group of $X$ is $\PGL_2(\bbC)^2 \rtimes C_2$.
Let $H = G \cap \PGL_2(\bbC)^2$ be a subgroup of index $1$ or $2$.
If $H$ is trivial, then $G$ is cyclic, $G$ has a fixed point and $X$
is $G$-unirational.
If $H$ has a fixed point, then the action of $H$ on each $\bbP^1$
must have a fixed point.  If $G=H$ then $X$ is
$G$-unirational; otherwise $H$ acts non-trivially on each $\bbP^1$ and
$X$ has exactly four $H$-fixed points.  Since $G$ takes $H$-fixed fibers to
$H$-fixed fibers, we see that $G$ has exactly $2$ fixed points.  Thus
$X$ is $G$-unirational.
It remains to consider the case where $H$ does not have fixed points.
Let $K$ be any subgroup of $H$ isomorphic to $C_2^2$ which does not have
fixed points on one of the copies of $\bbP^1$.
This group $K$ must exist since any action of a finite group on $\bbP^1$
without fixed points must contain a subgroup isomorphic to $C_2^2$.
There are $3$ such groups up to conjugacy in $\Aut(X)$, but they are all
birationally equivalent (see
Proposition~6.2.3~and~6.2.4~in~\cite{Bla06Finite}).
Obstruction A is one of these (equivalent) groups, so the theorem is
proved.
\end{proof}

The central idea for the case $d=5$ is contained in the proof of
Theorem~6.5(c)~of~\cite{BuhRei97On-the-essential}.
For completeness, we reproduce it here in more geometric language.

\begin{proof}[Proof of Theorem~\ref{thm:obsMain} for $d = 5$.]
Recall that the del Pezzo surface $X$ of degree $5$ is isomorphic to
$\overline{\mathcal{M}}_{0,5}$,
the moduli space of stable curves of genus 0 with 5 marked points.
We will show that $X$ is $G$-unirational for the entire automorphism group
$G \simeq S_5$.
There is an evident $S_5$-equivariant dominant rational map
$(\bbP^1)^5 \dasharrow X$ using the moduli space interpretation.
Noting that $(\bbP^1)^5$ is $S_5$-equivariantly birationally equivalent
to the standard permutation representation on $S_5$,
we see that $X$ is $S_5$-unirational.
\end{proof}

\section{Degree $4$}
\label{sec:deg4}

Let $X$ be a del Pezzo surface of degree $4$.
We recall a description of their structure and their automorphism groups
from Section 6.4 of \cite{DolIsk09Finite}.
In appropriate coordinates, they may be defined by the equations
\[
x_1^2 + \cdots + x_5^2 = a_1x_1^2 + \cdots + a_5x_5^2 = 0
\]
in $\bbP^4$ where $a_1,\ldots,a_5$ are distinct parameters.

The group of automorphisms of $\Pic(X)$ which preserve the intersection
form is isomorphic to the Weyl group $W(D_5)$ of the root system $D_5$.
We have $W(D_5) \simeq C_2^4 \rtimes S_5$.
The orthogonal complement $K_X^\perp$ of $K_X$ in $\Pic(X)$
is a lattice of rank $5$. 
The group $S_5$ permutes a basis of $K_X^\perp$ and elements of $C_2^4$
change the sign of an even number of those basis vectors.
The automorphism group of $X$ has a faithful action on $\Pic(X)$
thus we have $\Aut(X) \hookrightarrow W(D_5)$.

Let $N$ denote the normal subgroup of $W(D_5)$ isomorphic to $C_2^4$.
The action of $N$ can be realized on every surface $X$
by changing the sign of an even number of the coordinates $x_i$.
We will denote involutions in $N$ using the notation $\iota_A$ where $A$
is the subset of $\{1, \ldots, 5\}$ corresponding to the coordinates
whose sign is changed.  If $A$ is a subset with an odd
number of elements, then we define $\iota_A=\iota_{\bar{A}}$ where
$\bar{A}$ is the complement of $A$ in $\{1, \ldots, 5\}$.

Within the group $N$, an involution is of the \emph{first kind} if it
changes the sign of $4$ variables;
the involution $\iota_i$ fixes the elliptic curve $x_i=0$.
An involution is of the \emph{second kind} if it changes the sign of $2$
variables; the involution $\iota_{ij}$ where $i \ne j$ fixes a
set of $4$ points given by $x_i=x_j=0$.

\begin{lem} \label{lem:C22classification}
The conjugacy classes of subgroups of $N$ which are isomorphic to
$C_2^2$ are determined by the number of elements of each kind and are
listed in Table~\ref{tab:C22DP4} with the given interpretations.
\end{lem}
\begin{table}[h]
\begin{tabular}{|c|c|c|c|c|}
\hline
Type & 1st Kind & 2nd Kind & Interpretation & \cite{Bla06Finite} \\
\hline
I & 2 & 1 & has fixed points & C.22 \\
II & 1 & 2 & obstruction C & C.2,2 \\
III & 0 & 3 & obstruction A & P1.22.1 \\
\hline
\end{tabular}
\caption{Conjugacy classes of $N$ isomorphic to $C_2^2$}
\label{tab:C22DP4}
\end{table}

\begin{proof}
Since the action is diagonal in the basis given by $x_1, \ldots, x_5$,
one sees that these are the only possiblities by inspection.
The fixed points must correspond to eigenvalues of the group action,
and we conclude that only subgroups of Type I have fixed points on $X$.

For Type II, the involution of the first kind fixes pointwisely an
elliptic curve $E$, but the involutions of the second kind do not fix
any points on $E$.  By Lemma~\ref{lem:genus1fixed}, we conclude that
this must be obstruction C.

For Type III, we find that $\Pic(X)^G \simeq \bbZ^3$.
Thus $X$ is not minimal and we may equivariantly blow down
some exceptional curves.
Since the existence of a fixed point is an equivariant birational
invariant of a smooth surface, $X$ is equivariantly birational to a del
Pezzo surface of degree $\ge 5$ without fixed points.
The only possiblity is that $X$ has obstruction A.
\end{proof}

The possible splittings $S_5 \subset W(D_5)$ correspond to a choice of
geometric marking of $5$ skew exceptional curves.
Thus, any splitting $S_5 \subset W(D_5)$ has the following interpretation.
The surface $X$ is obtained by blowing up a set of $5$ points in general
position on $\bbP^2$.  These points sit on a unique conic $C$.  Each
automorphism of $C$ which preserves the $5$ points extends to an
automorphism of $\bbP^2$ and to the blowup $X$.
Note that not all subgroups of $S_5$ can be realized on every surface.
Indeed, the images of the map $\Aut(X) \to S_5$ can only be trivial,
$C_2$, $S_3$, $C_4$ or $D_{10}$.
Moreover, the involutions in the image of $\Aut(X) \to S_5$ are all
conjugate to the permutation $(12)(34)$.

\begin{lem} \label{lem:dp4fixedPt}
Let $X$ be a del Pezzo surface of degree $4$ with a faithful action of a
finite group $G$.
The group $G$ has a fixed point if and only if
$G$ does not contain a subgroup of Type II or III.
\end{lem}

\begin{proof}
Let $K$ be the kernel of $G \to S_5$ and $H$ be its image.

If $K$ has rank $3$ or $4$ then it must contain a subgroup of Type II or
III.  It remains to consider $K$ which is of Type I, is cyclic or is trivial.
In addition, we may assume $H$ is non-trivial.
We will show that all groups $G$ satisfying these conditions have a
fixed point.

If $K$ is of Type I then $K = \langle \iota_{4}, \iota_{5}\rangle$ after
a change of coordinates.
In this case, $H$ is isomorphic to $C_2$ or $S_3$.
Let $L$ be the preimage in $\Aut(X)$ of $H$ in $S_5$.
Let $K'$ be the group $\langle \iota_{12}, \iota_{23} \rangle$.
The group $K'$ is an $H$-invariant complement to $K$ in $N$
on which $H$ acts faithfully.

Since $K$ is normal in $L$, the set $S$ of $K$-fixed points on $X$ is
$L$-stable.
Explicitly, the set $S$ consists of the 4 points
\[ (\pm a:\pm b:\pm c:0:0) \]
for some values $a$, $b$, and $c$ in $\bbC$ and all possible sign
combinations.
Note that every non-trivial element of $K'$ acts on $S$ without fixed
point.
We have an embedding $L/K \hookrightarrow S_4$ where
$K'$ maps to the subgroup given by $\langle (12)(34), (13)(24) \rangle$.
Thus the image of $G$ in $L/K$ must act on the 4 points via an
involution conjugate to $(12)$ or as a subgroup isomorphic to $S_3$.
These always leave at least one point fixed in $S$.
We conclude that $G$ must fix a point on $X$ in this case.

We now consider the case where $K$ is cyclic.
If $K$ is an involution of the second kind, then again $H$ is
isomorphic to $C_2$ or $S_3$ and we may apply the reasoning above.

If $K$ is of the first kind, then there is an $G$-invariant genus 1
curve $E$ which is fixed pointwise by $K$.
The only possibilities for $H$ leaving $K$ invariant are cyclic groups
$C_2$, $C_3$ or $C_4$.
If $H \simeq C_3$ then $G$ is cyclic and we have a fixed point.
Otherwise, the only way for $G$ not to have a fixed point is for the
action of $H$ on $E$ to be translation.
If $H$ acts by translation then it contains an involution which acts by
translation since $H$ has order $2$ or $4$.
However, $i_{12}$ and $i_{13}$ do not fix points on $E$ and so must
generate the entire group $E[2]$ of $2$-torsion automorphisms.
Since these elements do not map non-trivially to $H$,
the group $H$ cannot act by translation.

It remains to consider $K$ trivial.  For $H$ cyclic, the surface $X$
must have a fixed point, so only $H \simeq S_3$ and $H \simeq D_{10}$
remain.
Since all elements of order 5 or 3 each form a single conjugacy class in
$W(D_5)$ we may assume $H$ contains $r=(12345)$ or $r=(134)$.
There is also an element in $H$ of the form $s=\iota_A(25)(34)$
where $A$ is a $(25)(34)$-invariant subset of $\{1,2,3,4,5\}$.
We require $srsr=1$, thus $\iota_Ar(\iota_A)=1$.
The subset $A=\emptyset$ is the only subset with an even number of
elements which is both $(25)(34)$-invariant and $r$-invariant.
Thus $G$ can be identified with a
subgroup of $S_5$ in $W(D_5)$.

Thus, $G$ permutes $5$ skew lines in $X$.  These can be blown
down to 5 points on a conic in $\bbP^2$.  The actions of $S_3$ or
$D_{10}$ on $\bbP^2$ fix a point outside of the invariant conic, so we
conclude that $X$ has a $G$-fixed point.
\end{proof}

Theorem~\ref{thm:obsMain} (and thus Theorem~\ref{thm:main}) in degree
$4$ is an immediate consequence of
Lemmas~\ref{lem:C22classification}~and~\ref{lem:dp4fixedPt}.
In fact, we obtain an additional characterization of $G$-unirationality
in this case:

\begin{cor}
Suppose $X$ is a del Pezzo surface of degree 4 with a faithful
$G$-action.
Then $X$ is $G$-unirational if and only if $X$ has a $G$-fixed point.
\end{cor}

\section{Degree $3$}
\label{sec:deg3}

We now consider del Pezzo surfaces of degree $3$; in other words, smooth
cubic surfaces in $\bbP^3$.  Throughout this section $X$ is a smooth
cubic surface with a faithful $G$-action.

The classification of automorphism groups of smooth cubic surfaces
can be found in Table~9.6~of~\cite{Dol12Classical}.
The set of isomorphism classes affording a given group of automorphisms
can be written down explicitly as families depending on parameters.
For certain special values of the parameters, the corresponding family
may acquire additional automorphisms.
In this case, we say the original automorphism
group $A$ \emph{specializes} to the larger automorphism group $B$ and
write $A \to B$.
(See also the discussion preceeding
Proposition~2.3~of~\cite{DolDun14Fixed}.)

For the convenience of the reader,
in Figure~\ref{fig:cubicSpec}, we list all the possible automorphism
groups of cubic surfaces along with their specializations.
Here $H_3(3)$ is the Heisenberg group of $3\times 3$ unipotent matrices
over $\bbF_3$.

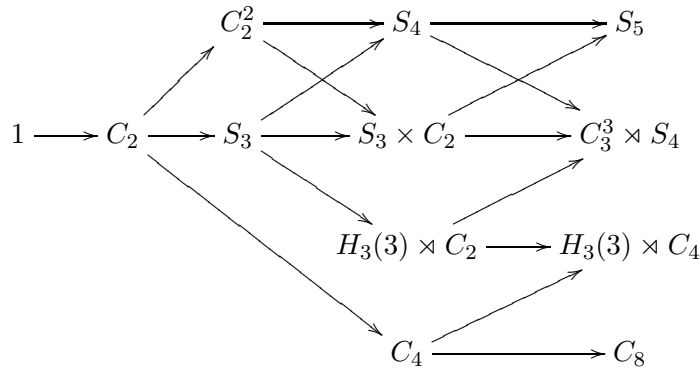
\begin{figure}[h]
\[
\xymatrix{
&
&
C_2^2 \ar[r] \ar[dr] &
S_4 \ar[r] \ar[dr] &
S_5 \\
1 \ar[r] &
C_2 \ar[ur] \ar[r] \ar[ddrr] &
S_3 \ar[r] \ar[ur] \ar[dr]&
S_3 \times C_2 \ar[r] \ar[ur] &
C_3^3 \rtimes S_4 \\
&
&
&
H_3(3) \rtimes C_2 \ar[r] \ar[ur]&
H_3(3) \rtimes C_4 \\
&
&
&
C_4 \ar[r] \ar[ur] &
C_8}
\]
\caption{Automorphism groups of cubic surfaces and their
specializations.}
\label{fig:cubicSpec}
\end{figure}

From Section~9.5.1~of~\cite{Dol12Classical},
the congugacy classes of cyclic subgroups in $W(E_6)$
all correspond to distinct conjugacy classes in $\PGL_4(\bbC)$
whenever they act on a cubic surface.
The number of elements from each conjugacy class in each family of
automorphisms can be found in Table 1 of~\cite{Hos97Automorphism}.
We will only be interested in the conjugacy classes of cyclic groups
of order $3$.  In the notation of~\cite{Car72Conjugacy},
the class $3A_2$ has eigenspaces of dimensions 1, 3;
the class $A_2$ has eigenspaces of dimensions 2,2; and
the class $2A_2$ has eigenspaces of dimensions 1,1,2.

A cubic surface $X$ is \emph{cyclic} if it is a triple cover of $\bbP^2$
branched over a smooth cubic curve $E$.
A cyclic surface can always be written in the form
\begin{equation} \label{eq:cyclic}
x_1^3 + x_2^3 + x_3^3 + x_4^3 + \alpha x_1x_2x_3 = 0
\end{equation}
in $\bbP^4$ for some parameter $\alpha$.
For a generic parameter $\alpha$,
the automorphism group is $\Aut(X) \simeq H_3(3) \rtimes C_2$.
For special values of $\alpha$ (for example $\alpha=0$), $X$ may have
additional automorphisms.

An important special case of the cyclic surfaces is the \emph{Fermat}
cubic surface given by
\begin{equation} \label{eq:Fermat}
x_1^3 + x_2^3 + x_3^3 + x_4^3 = 0
\end{equation}
in $\bbP^4$.
The automorphism group of the Fermat cubic surface is isomorphic to
$C_3^3 \rtimes S_4$ where the subgroup $C_3^3$ acts by multiplying the
coordinates $x_i$ by a primitive third root of unity $\epsilon$
and $S_4$ acts by permuting the
coordinates $\{ x_i \}$.

\begin{lem} \label{lem:C33classification}
The conjugacy classes of subgroups of $PGL_4(\bbC)$ which are isomorphic to
$C_3^2$ are listed in Table~\ref{tab:C33DP3}.
The table also lists the number of elements in each conjugacy class.
Each group acts only on smooth cubic surfaces with the given interpretation.
\end{lem}

\begin{table}[h]
\begin{tabular}{|c|ccc|c|c|}
\hline
Type & $3A_2$ & $2A_2$ & $A_2$ & Interpretation & \cite{Bla06Finite} \\
\hline
I & 4 & 2 & 2 & has fixed points & 3.33.1 \\
II & 2 & 6 & 0 & obstruction D & 3.33.2 \\
III & 0 & 4 & 4 & obstruction B & 0.V9 \\
\hline
\end{tabular}
\caption{Conjugacy classes $C_3^2$ in $PGL_4(\bbC)$}
\label{tab:C33DP3}
\end{table}

\begin{proof}
We begin with conjugacy in $\PGL_4(\bbC)$.
We follow the approach from pg. 500 of \cite{DolIsk09Finite}.
Since the kernel of the map $\SL_4(\bbC) \to \PGL_4(\bbC)$ has order
coprime to $3$, we may assume that $G$ has a preimage in $\SL_4(\bbC)$
consisting of diagonal matrices.
Consider the characters of $G$
acting on the given basis in $\bbF_3^4$.
Since the product of the characters must be trivial,
we may identify them with the space $\bbF_3^3$
viewed as an $S_4$-module.
There are 13 one-dimensional subspaces of $\bbF_3^3$
since $\bbP^2(\bbF_3)$ has 13 elements.
They have 3 orbits
represented by
$\langle(1,2,0,0)\rangle$,
$\langle(1,1,1,0)\rangle$ and
$\langle(1,1,2,2)\rangle$ in $\bbF_3^4$.
These correspond to Types I--III respectively.

We now describe how these groups act on a cubic surface $X$.
In each of these examples, $D(a,b,c,d)$ denotes a diagonal matrix with
entries $(a,b,c,d)$ and $\epsilon$ is a primitive third root of unity.

For Type I, the matrices $D(1,1,1,\epsilon)$ and $D(1,1,\epsilon,1)$
generate the group $G \simeq C_3^2$.
Here $X$ has up to three $G$-fixed points on the line $x_3=x_4=0$.

For Type II, the matrices
$D(\epsilon,\epsilon,\epsilon,1)$ and $D(1,\epsilon,\epsilon^2,1)$
generate the group $G \simeq C_3^2$.
The group $G$ acts on the variables $x_i$ via distinct characters
so the fixed points must be one of $(1:0:0:0),\ \ldots,\ (0:0:0:1)$.
The invariant monomials for this group action are
\[ x_1^3,\ x_2^3,\ x_3^3,\ x_4^3,\ x_1x_2x_3 \ . \]
If the resulting cubic surface is smooth then all but $x_1x_2x_3$
must appear in the corresponding cubic form.
Note that the point $(1:0:0:0)$ cannot lie on the surface since $x_1^3$
is the only monomial which is non-zero at this point.
Similarly, none of the other points lie on the surface and we conclude that
the action has no fixed points.
The matrix $D(\epsilon,\epsilon,\epsilon,1)$ fixes pointwisely a smooth
cubic curve and thus, by Lemma~\ref{lem:genus1fixed}, this group must be
obstruction D.

For Type III, the matrices
$D(\epsilon,\epsilon^2,1,1)$ and $D(1,1,\epsilon,\epsilon^2)$
generate the group $G \simeq C_3^2$.
The invariant monomials for this group action are
\[ x_1^3,\ x_2^3,\ x_3^3,\ x_4^3 \ . \]
and we conclude that there are no fixed points by the same argument as
Type II.
Let $\bbP^2$ have an action of $G$ without fixed points
(in other words, the surface from Obstruction B).
One checks that on $\bbP^2$ there are only 2 $G$-orbits consisting of
precisely 3 points each.
Blowing up these 6 points we obtain a cubic surface.
Since Type I and Type II have already been described,
we conclude that the surface of Type III has obstruction B.
\end{proof}

\begin{lem} \label{lem:3groupCubic}
If $G$ is a $3$-group with an action on a cubic surface $X$,
then either $G$ has a fixed point or it contains a subgroup
of Type II or III.
\end{lem}

\begin{proof}
The automorphism group of the Fermat cubic
has a Sylow $3$-subgroup $P$ isomorphic to $C_3^3 \rtimes C_3$
where the normal subgroup $C_3^3$ acts diagonally and
the quotient $C_3$ acts by permuting the first three basis vectors.
Recall that $W(E_6)$ has order $51840=2^7 \times 3^4 \times 5$.
Since $P$ has order $3^4$ and $\Aut(X) \hookrightarrow W(E_6)$,
any $3$-group acting on a cubic surface is isomorphic to a subgroup
of $P$.

We will show that all subgroups $G$ of $P$ are either cyclic, of Type I, or
contain a subgroup of Type II or III.
For a different surface, the particular embedding of $G$ into
$\PGL_4(\bbC)$ may be different, but the
conjugacy classes of all the elements must be the same.
Since cyclic groups and Type I groups always have fixed points,
this suffices to prove the lemma.

First, we note that if $G$ is $3$-elementary abelian then either $G$ is
cyclic, is of Type I--III, or has rank $3$.  If $G$ has rank $3$
then it contains subgroups of Type II and III.
Suppose $G$ is abelian with an element of order $9$.
Note that the only element of $C_3^3$ which is centralized by the
permutation action of $(123)$ is $D(\epsilon,\epsilon,\epsilon,1)$.
An element $g$ of order $9$ must be of the form $h(123)$ where $h$ is an
element of the diagonal group $C_3^3$.
Any element that commutes with $g$ must be a power of
$D(\epsilon,\epsilon,\epsilon,1)$, so we conclude that the cyclic group
$C_9$ is the only abelian group acting on a cubic surface that is not
$3$-elementary.

It remains to consider non-abelian groups.  Since the full group $P$
contains subgroups of types II and III, it suffices to consider
non-abelian groups of order $27$.
There are only 2 such groups, and each has a unique faithful $3$-dimensional
representation up to Galois conjugacy.

The first group has $3$-dimensional representation given by:
\[
\left\langle
g=\begin{pmatrix}
\zeta_9 & 0 & 0\\
0 & \zeta_9^4 & 0\\
0 & 0 & \zeta_9^7
\end{pmatrix},\
h=\begin{pmatrix}
0 & 1 & 0\\
0 & 0 & 1\\
1 & 0 & 0
\end{pmatrix}
\right\rangle
\]
where $\zeta_9$ is a root of unity.  
The subgroup $\langle g^3, h \rangle$ is of type II.

The other non-abelian group of order 27 is given by:
\[
\left\langle
g=\begin{pmatrix}
1 & 0 & 0\\
0 & \epsilon & 0\\
0 & 0 & \epsilon^2
\end{pmatrix},\
h=\begin{pmatrix}
0 & 1 & 0\\
0 & 0 & 1\\
1 & 0 & 0
\end{pmatrix}
\right\rangle \ .
\]
The subgroup $\langle g, ghg^{-1}h^{-1} \rangle$ is of type II.
\end{proof}

\begin{lem} \label{lem:S5unirational}
Suppose $X$ is a smooth cubic surface with an action of a finite group
$G$.
If $G$ is a subgroup of $S_5$ then $X$ is $G$-unirational.
\end{lem}

\begin{proof}
By Theorem~\ref{thm:cubicIndex2}, it suffices to consider groups $G$
which do not have subgroups of index $2$.  Thus, $G$ is one of
$C_3$, $C_5$, $A_4$ or $A_5$.  The cyclic groups have fixed points so
they are $G$-unirational.  The group $A_5$ only occurs on the Clebsch
cubic surface which is $A_5$-unirational by Hermite's theorem.
It remains to consider $G \simeq A_4$.

There is a unique faithful irreducible
representation $\sigma$ of $A_4$ of dimension $3$.
Fixing a basis $\{x_1,x_2,x_3\}$, the group is generated by the two maps
\begin{align*}
g \colon (x_1,x_2,x_3) &\mapsto (-x_1,-x_2,x_3)\\
h \colon (x_1,x_2,x_3) &\mapsto (x_2,x_3,x_1) \ .
\end{align*}
The action of $A_4$ on $\bbP^3$ is of the form $\sigma \oplus \chi$ where
$\chi$ is a $1$-dimensional representation with basis element $x_4$.
The semi-invariants of $\sigma$ of degree $\le 3$ are
\[ 1,\ x_1^2+x_2^2+x_3^2,\ x_1x_2x_3 \ . \]
We conclude that $X$ is of the form
\[ F = a x_4^3 + b x_4(x_1^2+x_2^2+x_3^2) + c x_1x_2x_3 \]
for some parameters $a,b,c$ and that $\chi$ must be the trivial
representation.

For general parameters $a,b,c$, the cubic surface $X$ is non-singular.
By construction, this family contains all smooth cubic $A_4$-surfaces.
In particular, this family contains the Clebsch cubic surface.
We obtain an irreducible family of configurations of 27 lines with an action of
$A_4$.
There are 6 skew lines invariant under $A_5$ on the Clebsch surface,
so they are \emph{a fortiori} invariant under $A_4$.
Since the family of configurations has an irreducible base, there is an
$A_4$-invariant set of 6 skew lines on any cubic $A_4$-surface in the
family of surfaces.
Thus, we have an $A_4$-equivariant birational equivalence with $\bbP^2$.
The $A_4$-action on $\bbP^2$ lifts to $\bbC^3$ so we conclude that
$X$ is $A_4$-unirational.
\end{proof}

Finally, we prove the main theorem for degree $d=3$.

\begin{proof}[Proof of Theorem~\ref{thm:obsMain} for $d = 3$.]
We simply need to prove that obstructions B and D
are the only obstructions to $G$-unirationality.

Consulting Figure~\ref{fig:cubicSpec}, we find that
all cubic $G$-surfaces fall into three (overlapping) cases:
either $G \subset S_5$, $G$ is a cyclic group, or $X$ is a cyclic surface.
In the first two cases, the surface is
$G$-unirational unconditionally.
The first case follows from Lemma~\ref{lem:S5unirational}.
The second case follows since every cyclic group has a fixed point.

It remains only to consider the case where $X$ is a cyclic surface.
First, we consider the case where $X$ is not the Fermat cubic.
In this case $\Aut(X) \simeq H_3(3) \rtimes \langle g \rangle$
where $g$ is an element of order $2$ or $4$.
In either case, by Theorem~\ref{thm:cubicIndex2} we may assume
$G \subset H_3(3)$.
The implication now follows from Lemma~\ref{lem:3groupCubic}
since all cyclic and Type I groups have fixed points.

It remains to consider the Fermat cubic.  We have an exact sequence
\[ 1 \to K \to G \to H \to 1 \]
where $K$ is a subgroup of $C_3^3$ and $H$ is a subgroup of $S_4$.
By Theorem~\ref{thm:cubicIndex2}, it suffices to assume that
$H$ is trivial, $C_3$ or $A_4$.
We may again handle the cases when $H$ is trivial or $C_3$
via Lemma~\ref{lem:3groupCubic}.

Only the case $H \simeq A_4$ remains.
The only $A_4$-invariant subgroups $K$ of $C_3^3$ are the trivial group
and the full group $C_3^3$.
If $K$ is trivial then $X$ is $G$-unirational by
Lemma~\ref{lem:S5unirational}.
If $K$ is $C_3^3$ then $G$ contains a subgroup of Type II.
\end{proof}

\subsection*{Acknowledgements}

The author would like to thank I.~Dolgachev for useful discussions, and
Z.~Reichstein for helpful comments, especially a simplification of
the proof of Theorem~\ref{thm:cubicIndex2}.
The author would also like to thank the anonymous referees for useful
comments.

\end{document}